\documentclass[STIX1COL]{WileyNJD-v2}

\usepackage{tabularx,booktabs,ifsym}
\usepackage{pstricks, pst-3d, pst-plot,tikz}
\usepackage{amsmath}
\usepackage[numbers,sort&compress]{natbib}
\usepackage{accents}
\usepackage{enumerate}
\usepackage{amsmath}
\usepackage{amsfonts}
\usepackage{caption,cleveref}
\crefname{equation}{}{}
\usepackage{amsfonts}
\usepackage{array}
\usepackage{wasysym}
\usepackage{mathtools}
\usepackage{mathdots}
\usepackage{graphicx,epsfig,enumerate}
\usepackage{tabularx,subfigure}
\usepackage{moreverb}
\usepackage{multirow}
\usepackage{algorithm}  

\newcommand{\IR}{{\mathbb{R}}}

\newcommand{\CA}{{\mathcal{A}}}
\newcommand{\CB}{{\mathcal{B}}}

\newcommand{\CE}{{\mathcal{E}}}

\newcommand{\bi}{{\bf i}}

\newcommand{\br}{{\bf r}}

\newcommand{\bs}{{\bf s}}

\newcommand{\by}{{\bf y}}
\newcommand{\bx}{{\bf x}}

\newcommand{\diag}{{\rm diag}}
\newcommand{\blkdiag}{{\rm blkdiag}}

\definecolor{red}{rgb}{1,0,0}
\definecolor{green}{rgb}{0,1,0}
\definecolor{blue}{rgb}{0,0,1}
\definecolor{cyan}{cmyk}{1,0,0,0}
\definecolor{magenta}{cmyk}{0,1,0,0}
\definecolor{yellow}{cmyk}{0,0,1,0}
\newcommand{\Cb}[1]{{\textcolor{blue}{#1}}}

\newcommand{\Cr}[1]{{\textcolor{red}{#1}}}

\articletype{Article Type}%

\received{26 April 2016}
\revised{6 June 2016}
\accepted{6 June 2016}

\begin{document}

\title{Some Tucker-like approximations based on the modal semi-tensor product \protect
	\thanks{The first author was supported by the National Natural Science Foundation of China (No. 11801074), key research projects of general universities in Guangdong Province (No. 2019KZDXM034), basic research and applied basic research projects in Guangdong Province (Projects of Guangdong, Hong Kong and Macao Center for Applied Mathematics) (No. 2020B1515310018), and the Shantou University Start-up Funds for Scientific Research. The second author was supported by research grants MYRG2019-00042-FST from University of Macau and 0014/2019/A from FDCT of Macao. The third author was supported by the Zhejiang Provincial Natural Science Foundation of China under Project LY21A010010 and Project LY21A010004, the Fundamental Research Funds for the Provincial Universities of Zhejiang GK199900299012-006.}}

\author[1]{Ze-Jia Xie}

\author[2]{Xiao-Qing Jin*}

\author[3]{Zhi Zhao}

\authormark{ZE-JIA XIE \textsc{et al}}

\address[1]{\orgdiv{Department of Mathematics, College of Science}, \orgname{Shantou University}, \orgaddress{\state{Shantou, 515063}, \country{China}}}

\address[2]{\orgdiv{Department of Mathematics}, \orgname{University of Macau}, \orgaddress{\state{Macao}, \country{China}}}

\address[3]{\orgdiv{Department of Mathematics, School of Sciences}, \orgname{Hangzhou Dianzi University}, \orgaddress{\state{Hangzhou, 310018}, \country{China}}}

\corres{*Xiao-Qing Jin, Department of Mathematics, University of Macau, Macao, China. \email{xqjin@umac.mo}}


\abstract[Summary]{ Approximating higher-order tensors by the Tucker format has been applied in many fields such as psychometrics, chemometrics, signal processing, pattern classification, and so on. In this paper, we propose some new Tucker-like approximations based on the modal semi-tensor product (STP), especially, a new singular value decomposition (SVD) and a new higher-order SVD (HOSVD) are derived. Algorithms for computing new decompositions are provided. We also give some numerical examples to illustrate our theoretical results.}

\keywords{ Tucker approximation, semi-tensor product, HOSVD, SVD}


\maketitle

\footnotetext{\textbf{Abbreviations:} STP, semi-tensor product; SVD, singular value decomposition; HOSVD, higher-order singular value decomposition}

\section{Introduction}\label{sec:intro}
The Tucker format approximation of higher-order tensors is a useful tensor approximation. It has been applied in areas of psychometrics \cite{KiMe01}, chemometrics \cite{He94,SmBrGe04}, computer vision \cite{VaTe02},  signal processing \cite{LaVa04,MuBo05}, data mining \cite{SaEl07}, and so on. Several algorithms have been developed to compute the Tucker approximation, including the higher-order singular value decomposition (HOSVD) \cite{LaMoVa20}, the higher-order orthogonal iterations (HOOI) \cite{KrLe80,LaMoVa20_2}, the Newton-type algorithms \cite{ElSa09,IsLaAbHu09,SaLi10}, the Riemannian trust-region method \cite{IsAbHuLa11}, and the randomized algorithms \cite{ChWe19,ChWeYa20,DrMa07}. A review of tensor decompositions including the Tucker decomposition can be found in \cite{KoBa09}.

In this paper, we consider a new Tucker-like approximation based on the modal semi-tensor product (STP) defined in Section~\ref{sec:mstp}. The STP of matrices was firstly proposed in 2001 by Cheng \cite{Ch01}. Since then it has received many applications, and become an important tool in stabilization and control design of dynamic systems, e.g., power systems \cite{MeLiXu10}, analysis and control of logical systems \cite{ChQiLi11}, finite games \cite{ChHeQi01}, etc. The STP is a generalization of conventional matrix multiplication. For the matrix multiplication $A\times B$, the column number of $A$ must be equal to the row number of $B$, which is called the ``equal dimension'' condition, while, for the STP of matrices $A$ and $B$, denoted by $A\ltimes B$, we only require that $A$ and $B$ satisfy the ``multiple dimension'' condition, i.e., the column number of $A$ is a multiple or a factor of the row number of $B$. If such multiple is equal to 1, then the STP becomes the conventional matrix multiplication.

Similar to the tensor mode-$k$ product, i.e., multiplying a tensor by a matrix in mode $k$, we can also develop a new mode-$k$ product between a tensor and a matrix based on the STP. Then a new Tucker-like approximation based on this new mode-$k$ product can be proposed. We develop new versions of SVD and HOSVD for achieving such approximation. It is well-known that the computation of the SVD for an $n$-by-$n$ matrix requires $\mathcal{O}(n^3)$ operations and at least $\mathcal{O}(n^2)$ storage. Thus, if the dimension $n$ is very large, then the computation of the SVD is often infeasible. In this case other type of approximations are considered. The SVD based on the STP could be a possible approximation to deal with this case.

\section{Notation and preliminaries}
\label{sec:pre}

\subsection{Tensor unfolding}
\label{ssec:unfolding}

The tensor unfolding, also known as the matricization \cite{GoVa13,LaMoVa20} or flattening \cite{VaTe02}, is to arrange the entries of a tensor into a matrix. The tensor unfolding plays an important role in tensor computations because it can serve as a bridge between matrix computations and tensor computations. There are different kinds of tensor unfoldings. Particularly important are the modal unfoldings.

Let $\CA \in \IR^{n_1\times \cdots \times n_d}$, $n=n_1 \cdots n_d$, and $\bi =[i_1,\dots,i_d]$ with $1\leq i_k \leq n_k$ for $1\leq k \leq d$. We define a multi-index by 
\begin{equation}\label{<>}
	<\bi>:=i_1+\sum_{\alpha=2}^{d}(i_\alpha-1)\prod_{\beta=1}^{\alpha-1}n_\beta.
\end{equation}
The mode-$k$ unfolding of $\CA$ is an $n_k$-by-$(n/n_k)$ matrix, denoted by $\CA_{(k)}$, whose columns are the mode-$k$ fibers. Precisely, the mode-$k$ unfolding $\CA_{(k)}$ is defined by 
\begin{equation}
	\CA_{(k)}(i_k, <\bi_{-k}>)=\CA(\bi),
\end{equation}
where $\bi_{-k}=[i_1,\dots,i_{k-1},i_{k+1},\dots,i_d]$ and $<\bi_{-k}>$ is defined as in $(\ref{<>})$.

\subsection{STP of matrices}
\label{ssec:STP}
The following definition for the new inner product is essential for the definition of the STP of matrices.

\begin{definition}[\cite{Ch01}]
	Let $\bx=[x_1,\dots,x_s]^\top\in \IR^s$ and $\by=[y_1,\dots,y_t]^\top \in \IR^t$. 
	
	Case 1: If $t$ is a factor of $s$, i.e., if $s=t\cdot n$ for some positive integer $n$, then we split $\bx^\top$ into $t$ equal blocks, named $\bx_1^\top,\dots,\bx_t^\top$. Each block is an $n$-dimensional row vector. The (left) STP of $\bx^\top$ and $\by$ is an $n$-dimensional row vector defined as 
	\begin{equation}
		\bx^\top \ltimes \by := \sum_{k=1}^{t}y_k\bx_k^\top \in \IR^{1\times n}.
	\end{equation}
	
	Case 2: If $s$ is a factor of $t$, i.e., if $t=s\cdot n$ for some positive integer $n$, then we split $\by$ into $s$ equal blocks, named $\by_1,\dots,\by_s$. Each block is an $n$-dimensional column vector. The (left) STP of $\bx^\top$ and $\by$ is an $n$-dimensional column vector defined as 
	\begin{equation}
		\bx^\top \ltimes \by := \sum_{k=1}^{s}x_k\by_k \in \IR^{n}.
	\end{equation}
	\label{def:stp0}
\end{definition}
In Definition~\ref{def:stp0}, if $t=s$, the inner product based on STP becomes the conventional inner product $\bx^\top \by$. Examples for Cases 1 and 2 are showed as follows, respectively:

\begin{align*}
	\begin{bmatrix}\Cb{x_1}, \Cb{x_2}, \Cr{x_3}, \Cr{x_4}, x_5, x_6\end{bmatrix} \ltimes 
	\begin{bmatrix}
		\Cb{y_1} \\ \Cr{y_2} \\ y_3
	\end{bmatrix} & =
	\Cb{y_1} \left[\Cb{x_1}, \Cb{x_2}\right]  + \Cr{y_2} \left[\Cr{x_3}, \Cr{x_4}\right] + y_3 \left[x_5, x_6\right] \\ 
	& = \begin{bmatrix}\Cb{y_1x_1}+\Cr{y_2x_3}+y_3x_5, \Cb{y_1x_2}+\Cr{y_2x_4}+y_3x_6\end{bmatrix}
\end{align*}
and
\begin{align*}
	\begin{bmatrix}
		\Cb{x_1}, \Cr{x_2}, x_3
	\end{bmatrix}  \ltimes 
	\begin{bmatrix}\Cb{y_1} \\ \Cb{y_2}\\ \Cr{y_3}\\ \Cr{y_4}\\ y_5\\ y_6\end{bmatrix}	& =
	\Cb{x_1}	\begin{bmatrix}\Cb{y_1} \\ \Cb{y_2}\end{bmatrix} + 
	\Cr{x_2} \begin{bmatrix}\Cr{y_3} \\ \Cr{y_4}\end{bmatrix} + 
	x_3 \begin{bmatrix} y_5 \\y_6\end{bmatrix} \\ 
	& = \begin{bmatrix}\Cb{x_1y_1}+\Cr{x_2y_3}+x_3y_5 \\ \Cb{x_1y_2}+\Cr{x_2y_4}+x_3y_6\end{bmatrix}.
\end{align*}

\begin{definition}[\cite{ChZh03}]
	Let $M\in \IR^{m\times n}$ and $N\in \IR^{p\times q}$. If $n$ is a factor of $p$ or $p$ is a factor of $n$, then $C=M \ltimes N$ is called the (left) STP of $M$ and $N$, where $C$ consists of $m\times q$ blocks as $C=[C^{ij}]$ and
	\begin{equation}
		C^{ij}=M(i,:) \ltimes N(:,j), \ \ i=1,\dots,m,\  j=1,\dots,q.
	\end{equation}
	\label{def:stp}
\end{definition}

In Definition \ref{def:stp}, if $n=p$, the left STP becomes the conventional matrix product $MN$. Cheng and Zhang \cite{ChZh03} also proposed the right STP for matrices. However, since the left STP has better properties than the right one, we only use the left STP in this paper, and STP is defaulted to be the left STP.

\begin{remark}
	Let $A\in \IR^{m\times n}$ and $\bx\in \IR^p=\IR^{p\times 1}$. If $n$ is a factor of $p$, say $p=t\cdot n$ for some positive integer $t$, then $A\ltimes \bx\in \IR^{tm}$ which is still a column vector. If $p$ is a factor of $n$, say $n=s\cdot p$ for some positive integer $s$, then $A\ltimes \bx\in \IR^{m\times s}$ which is an $m$-by-$s$ matrix.

	For a higher dimension vector $\bx\in \IR^p$ with  $p=t\cdot n$ for some positive integer $t$, we want to use the STP to reduce its dimension, i.e., the dimension of $A\ltimes \bx\in \IR^{mt}$ is much smaller than the dimension of $\bx\in \IR^p$, and we hope that the product $A\ltimes \bx$ is still a column vector rather than a matrix. Based on this consideration, we focus on the case that $n$ is a factor of $p$ in Definition~\ref{def:stp}.
\end{remark}

Let ``$\otimes$'' denote the Kronecker product of matrices \cite{GoVa13} and $I_n$ denote the $n$-by-$n$ identity matrix. The following proposition shows some properties of the STP.

\begin{proposition}[\cite{ChZh03}]
	(i) If $A\in \IR^{m\times np}$ and $B\in \IR^{p\times q}$, then 
	\begin{equation}
		A\ltimes B=A(B\otimes I_n) \in \IR^{m\times nq}.
	\end{equation}
	
	(ii) If $A\in \IR^{m\times n}$ and $B\in \IR^{np\times q}$, then 
	\begin{equation}
		A\ltimes B=(A\otimes I_p)B \in \IR^{mp\times q}.
	\end{equation}
	
	(iii) Suppose that the matrices $A, B, C$ have proper dimensions such that $\ltimes$ is well defined. Then
	\begin{equation}
		(A\ltimes B)\ltimes C=A\ltimes (B\ltimes C). 
	\end{equation}
	\label{pro:iii}
\end{proposition}

\section{The modal STP}
\label{sec:mstp}

Firstly, we recall the conventional tensor modal product \cite{GoVa13}. Suppose that $\CA\in \IR^{n_1 \times \cdots \times n_d}$ and $U\in \IR^{m\times n_k}$ for $1\leq k\leq d$. Then the tensor $\CB \in \IR^{n_1\times \cdots n_{k-1}\times m\times n_{k+1}\times \cdots \times n_d}$ is the mode-$k$ product of $\CA$ and $U$ if
\begin{equation}
	\CB_{(k)}=U\cdot \CA_{(k)}
\end{equation}
or elementwise
\begin{equation*}
	\CB(i_1,\dots,i_{k-1},j,i_{k+1},\dots,i_d)=\sum_{i_k=1}^{n_k}U(j,i_k)\CA(i_1,\dots,i_{k-1},i_k,i_{k+1},\dots,i_d).
\end{equation*}
This operation is denoted by 
$$\CB=\CA \times_k U.$$

Next, we propose a new modal product based on the STP introduced in Definition~\ref{def:stp}.

\begin{definition}\label{def:mstp}
	Suppose that $\CA\in \IR^{n_1 \times \cdots \times n_d}$ and $U\in \IR^{m\times (n_k/s_k)}$ where $s_k$ is a factor of $n_k$ for $1\leq k\leq d$. Then $\CB \in \IR^{n_1\times \cdots n_{k-1}\times s_k m\times n_{k+1}\times \cdots \times n_d}$ is the mode-$k$ STP of $\CA$ and $U$ if
	\begin{equation}
		\CB_{(k)}=U\ltimes \CA_{(k)}.
	\end{equation}
	The result, denoted by 
	$$\CB=\CA\ltimes_k U,$$
	consists of $n_1\times \cdots \times n_{k-1}\times m\times n_{k+1}\times \cdots \times n_d$ blocks as $\CB=[\CB^{i_1\dots i_{k-1}ji_{k+1}\dots i_d}]$ with each block
	\begin{equation}
		\CB^{i_1\dots i_{k-1}ji_{k+1}\dots i_d}=U(j,:)\ltimes \CA(i_1,\dots,i_{k-1},:,i_{k+1},\dots,i_d)\in \IR^{s_k}
	\end{equation}
	for $1\leq i_\alpha\leq n_\alpha$, $1\leq \alpha\leq d$, $\alpha\neq k$, and $1\leq j\leq m$.
\end{definition}
By Proposition~\ref{pro:iii} (i), it is obvious that 
\begin{equation}\label{eq:timesI}
	\CB=\CA\ltimes_k U=\CA \times_k (U\otimes I_{s_k}).
\end{equation}

Similar to the properties of modal product (see Properties 2 and 3 in \cite{LaMoVa20}), the modal STP also has the following properties.
\begin{lemma}\label{le:stp-property}
	Suppose that $\CA\in \IR^{n_1 \times \cdots \times n_d}$ and matrices $U, V, W$ have proper dimensions such that the modal STPs $\ltimes_s, \ltimes_t$, and $\ltimes$ are well defined. Then 
	\begin{enumerate}[(i)]
		\item \ \  $ (\CA \ltimes_s U)\ltimes_t V = (\CA \ltimes_t V) \ltimes_s U$\ \ for $1\leq s \neq t \leq d$;
		\item \ \  $(\CA \ltimes_t V)\ltimes_t W = \CA \ltimes_t (W\ltimes V)$.
	\end{enumerate}
\end{lemma}  

\begin{proof}
Suppose that the matrices $U, V, W$ and the identity matrices $I_1, I_2, I_3$ have proper dimensions. Let $I_3=I_3^\prime \otimes I_2$.

(i) We have from \cref{eq:timesI} and the Property 2 in \cite{LaMoVa20} that 
\begin{align*}
	(\CA \ltimes_s U) \ltimes_t V &= (\CA \times_s (U\otimes I_1)) \times_t (V \otimes I_2)\\
	& =(\CA \times_t(V\otimes I_2))\times_s (U\otimes I_1) \\
	& =(\CA \ltimes_t V) \ltimes_s U,
\end{align*}
which proves the first equality.

(ii) Similarly, it follows from \cref{eq:timesI} and the Property 3 in \cite{LaMoVa20} that 
\begin{align*}
	(\CA \ltimes_t V) \ltimes_t W & = (\CA \times_t (V\otimes I_2))\times_t (W\otimes I_3)\\
	& = \CA \times_t [(W\otimes I_3)(V\otimes I_2)]\\
	& = \CA \times_t [(W\otimes I_3^\prime \otimes I_2)(V\otimes I_2)]\\
	& = \CA \times_t [(W\otimes I_3^\prime)V\otimes I_2]\\
	& = \CA \ltimes_t [(W\otimes I_3^\prime)V]\\
	& = \CA \ltimes_t (W\ltimes V),
\end{align*}
which shows that the second equality is true.  
\end{proof}

\begin{proposition}\label{pro:I}
	Let $\CA\in \IR^{n_1 \times \cdots \times n_d}$ and $s_k$ be a factor of $n_k$ for $1\leq k\leq d$. The matrix $I_{n_k/s_k}$ denotes the identity matrix with dimension $n_k/s_k$. Then 
	\begin{equation}
		\CA \ltimes_k I_{n_k/s_k} = \CA.
	\end{equation}
\end{proposition}

\begin{proof}
It follows from Definition \ref{def:mstp} and Proposition \ref{pro:iii} (ii) that
\begin{equation*}
	(\CA\ltimes_k I_{n_k/s_k})_{(k)} = I_{n_k/s_k}\ltimes \CA_{(k)} = (I_{n_k/s_k}\otimes I_{s_k})\CA_{(k)}=I_{n_k}\CA_{(k)}=\CA_{(k)}.
\end{equation*}
Therefore, $\CA\ltimes_k I_{(n_k/s_k)}= \CA$. 
\end{proof}

\begin{proposition}\label{pro:T}
	Suppose that tensors $\CA\in \IR^{n_1 \times \cdots \times n_d}$, $\CB\in \IR^{m_1 \times \cdots \times m_d}$, and $U^{(k)}\in \IR^{(n_k/s_k)\times (m_k/s_k)}$ satisfying $U^{(k)\top}U^{(k)}=I_{m_k/s_k}$ for $1\leq k\leq d$. If 
	$$\CA=\CB \ltimes_1 U^{(1)}\ltimes_2 U^{(2)}\ltimes_3 \cdots \ltimes_d U^{(d)},$$
	then 
	\begin{equation}\label{eq:ortho}
		\CB = \CA\ltimes_1 U^{(1)\top}\ltimes_2 U^{(2)\top}\ltimes_3 \cdots \ltimes_d U^{(d)\top}.
	\end{equation}
\end{proposition}

\begin{proof}
We first prove that if $\CA = \CB \ltimes_k U^{(k)}$, and $U^{(k)\top}U^{(k)}=I_{m_k/s_k}$, then $\CB = \CA \ltimes_k U^{(k)\top}$.

We have from Proposition \ref{pro:I} and Lemma \ref{le:stp-property} (ii) that
$$\CA \ltimes_k U^{(k)\top}=(\CB \ltimes_k U^{(k)}) \ltimes_k U^{(k)\top}=\CB\ltimes_k (U^{(k)\top} U^{(k)})=\CB \ltimes_k I_{m_k/s_k}=\CB.$$
Hence, by Lemma \ref{le:stp-property} (i), (\ref{eq:ortho}) can be obtained.          
\end{proof}

The following proposition can be obtained by Definition \ref{def:mstp} and the fact (\ref{eq:timesI}) combining Proposition 3.7 in \cite{Ko06} for the conventional tensor modal product.
\begin{proposition}\label{pro:unfold}
	Given a tensor $\CB\in \IR^{m_1 \times \cdots \times m_d}$ and a series of matrices $U^{(k)}\in \IR^{(n_k/s_k)\times (m_k/s_k)}$ for $1\leq k\leq d$, suppose that 
	$$\CA=\CB \ltimes_1 U^{(1)}\ltimes_2 U^{(2)}\ltimes_3 \cdots \ltimes_d U^{(d)} \in \IR^{n_1 \times \cdots \times n_d},$$
	then 
	\begin{equation}\label{eq:unfolding}
		\CA_{(k)}=\widehat{U^{(k)}}\, \CB_{(k)}\, (\widehat{U^{(d)}}\otimes \cdots \otimes \widehat{U^{(k+1)}}\otimes \widehat{U^{(k-1)}} \otimes \cdots \otimes \widehat{U^{(1)}})^\top
	\end{equation}
	and
	\begin{equation}\label{eq:vec}
		{\rm vec}(\CA)= (\widehat{U^{(d)}}\otimes \widehat{U^{(d-1)}}  \otimes \cdots \otimes \widehat{U^{(1)}}){\rm vec}(\CB),
	\end{equation}
\end{proposition}
where $\widehat{U^{(k)}}:=U^{(k)} \otimes I_{s_k}$ for $1\leq k\leq d$.

\section{New Tucker-like approximations}
\label{sec:tucker}
\subsection{The second-order case: SVD}
First, we consider the second-order case of Tucker-like approximations. It is well-known that the best rank-$r$ approximation of a given matrix $A\in \IR^{m\times n}$ can be obtained by a truncated SVD of $A$, i.e., 
\begin{equation}
	A = U_r \Sigma_r V_r^\top+E,
\end{equation}
where $E \in \IR^{m\times n}$ is the approximation error, $\Sigma_r={\rm diag} (\sigma_1, \dots, \sigma_r)\in \IR^{r \times r}$ is a diagonal matrix containing the $r$ largest singular values $\sigma_1 \geq \cdots \geq \sigma_r$ of $A$, and $U_r \in \IR^{m\times r}$, $V_r \in \IR^{n\times r}$ are matrices whose columns are the leading $r$ left and right singular vectors of $A$, respectively. In other words, a rank-$r$ approximation of $A$ that minimizes $\|A-B\|_F$ is given by
\begin{equation*}
	B=U_r \Sigma_r V_r^\top=\Sigma_r \times_1 U_r \times_2 V_r.
\end{equation*}
This is a special case of the Tucker approximation. Next, an SVD-like approximation of a given matrix based on the STP is considered. We call it the SVD-STP. Let 
\begin{equation}\label{PA}
	A=
	\begin{bmatrix}
		A_{1,1} & \cdots & A_{1,n_1}\\
		\vdots & \ddots & \vdots \\
		A_{m_1,1} & \cdots & A_{m_1,n_1}\\
	\end{bmatrix}\in \IR^{m_1m_2 \times n_1n_2},
\end{equation}
where each $A_{i,j}\in \IR^{m_2\times n_2}$. Then
$\widetilde{A}\in \IR^{m_1n_1\times m_2n_2}$ is defined by 
\begin{equation}\label{eq:RA}
	\widetilde{A}=
	\begin{bmatrix}
		\overline{A}_{1} \\
		\vdots \\
		\overline{A}_{n_1}\\
	\end{bmatrix}\ \ 
	\text{with}\ \ 
	\overline{A}_j=
	\begin{bmatrix}
		{\rm vec}(A_{1,j})^\top \\
		\vdots \\
		{\rm vec}(A_{m_1,j})^\top\\
	\end{bmatrix}.
\end{equation}
The following lemma is useful for computing the SVD-STP. 

\begin{lemma}[\cite{GoVa13,VaPi93}]\label{le-ksvd}
	Let $A\in \IR^{m\times n}$ with $m=m_1m_2$ and $n=n_1n_2$. If $\widetilde{A}$ defined by (\ref{eq:RA}) has the SVD 
	\begin{equation}
		U^\top \widetilde{A} V= \Sigma = {\rm diag} (\tilde{\sigma}_1, \dots, \tilde{\sigma}_p),
	\end{equation} 
	where $p=\min\{m_1n_1,m_2n_2\}$, $\tilde{\sigma}_1$ is the largest singular value of $\widetilde{A}$, and $U(:,1)$, $V(:,1)$ are the corresponding left and right singular vectors, respectively, then the matrices $B\in \IR^{m_1\times n_1}$ and $C\in \IR^{m_2\times n_2}$ defined by ${\rm vec}(B)=\sqrt{\tilde{\sigma}_1} U(:,1)$ and ${\rm vec}(C)=\sqrt{\tilde{\sigma}_1}V(:,1)$ minimize 
	$$
	\|A-B\otimes C\|_F.
	$$ 
\end{lemma}

\begin{theorem}[SVD-STP]\label{thm:fullsvd} 
	Let $A\in \IR^{n_1\times n_2}$ and $s_1$, $s_2$ be factors of $n_1$, $n_2$, respectively. Then there exist orthogonal matrices $U\in \IR^{(n_1/s_1)\times (n_1/s_1)}$ and $V\in \IR^{(n_2/s_2)\times (n_2/s_2)}$  such that 
	\begin{equation}\label{eq:gsvd}
		A = U\ltimes \Sigma \ltimes V^\top + E_1=\Sigma \ltimes_1 U\ltimes_2 V + E_1,
	\end{equation}
	where $\Sigma=\blkdiag (S_{1},S_{2},\dots, S_{p}) \in \IR^{n_1\times n_2}$ $(p=\min\{n_1/s_1,n_2/s_2\})$ is a rectangular block-diagonal matrix defined by aligning $S_{1},\dots,S_{p}\in \IR^{s_1\times s_2}$ along the diagonal of $\Sigma$ with $\|S_{1}\|_F\geq \|S_{2}\|_F \geq \cdots \geq \|S_{p}\|_F$, and $E_1$ represents the approximation error.
\end{theorem}

\begin{proof}
Given a matrix $A\in \IR^{n_1\times n_2}$ and factors $s_1$, $s_2$, let $B\in \IR^{(n_1/s_1)\times (n_2/s_2)}$ and $C\in \IR^{s_1\times s_2}$, obtained by Lemma~\ref{le-ksvd}, be the solution to the minimization problem
\begin{equation*}
	\min_{B,C} \|A-B \otimes C\|_F.
\end{equation*}
We write 
\begin{equation*}
	A = B \otimes C + E_1,
\end{equation*}
where $E_1 \in \IR^{n_1\times n_2}$ is the approximation error. Let $B=U\Sigma_B V^\top$ be the full SVD of $B$, where $U\in \IR^{(n_1/s_1)\times (n_1/s_1)}$ and $V\in \IR^{(n_2/s_2)\times (n_2/s_2)}$ are orthogonal matrices, and
$$\Sigma_B=\diag (\sigma_1,\sigma_2,\dots,\sigma_p)\in\IR^{(n_1/s_1)\times (n_2/s_2)}$$
with $\sigma_1\geq \sigma_2\geq \dots\geq \sigma_p\geq 0$ and $p=\min\{n_1/s_1,n_2/s_2\}$, is a rectangular diagonal matrix with the singular values of $B$ on the diagonal. Write $C=I_{s_1}CI_{s_2}^\top$. It follows from the properties of Kronecker product that 
\begin{align*}
	A &= B\otimes C +E_1 =(U\otimes I_{s_1})(\Sigma_B \otimes C)(V\otimes I_{s_2})^\top + E_1 \\
	& =U\ltimes \Sigma \ltimes V^\top + E_1 =\Sigma \ltimes_1 U\ltimes_2 V + E_1,
\end{align*}
where $\Sigma=\Sigma_B \otimes C= \blkdiag (S_{1},S_{2},\dots, S_{p})\in \IR^{n_1\times n_2}$ is a rectangular block-diagonal matrix defined by aligning the matrices $S_{1},S_{2},\dots,S_{p}$ along the diagonal of $\Sigma$ with each block $S_{k}=\sigma_k C\in \IR^{s_1\times s_2}$, $1\leq k\leq p$, and obviously $\|S_{1}\|_F\geq \|S_{2}\|_F \geq \cdots \geq \|S_{p}\|_F$.  
\end{proof}

 The following algorithm is given for constructing the approximation (\ref{eq:gsvd}).

\begin{algorithm}
	\caption{SVD-STP}
	\label{alg:FSVD-STP}
	\begin{algorithmic}
		\State Given $A\in \IR^{n_1\times n_2}$ and factors $s_1$, $s_2$.
		\State (1) Compute matrices $B\in \IR^{(n_1/s_1)\times (n_2/s_2)}$ and $C\in \IR^{s_1\times s_2}$, which minimize $\|A-B\otimes C\|_F$, by using Lemma~\ref{le-ksvd}.
		\State (2) Compute the SVD of $B$, i.e., $B=U\Sigma_B V^\top$, where $U\in \IR^{(n_1/s_1)\times (n_1/s_1)}$, $V\in \IR^{(n_2/s_2)\times (n_2/s_2)}$, and $\Sigma_B\in\IR^{(n_1/s_1)\times (n_2/s_2)}.$
		\State (3) Compute $\Sigma=\Sigma_B\otimes C\in \IR^{n_1\times n_2}$.
		
		\Return  $\Sigma, U, V$
	\end{algorithmic}
\end{algorithm}

If $s_1=s_2=1$, the approximation in (\ref{eq:gsvd}) becomes the SVD of matrices. Besides, it follows from Lemma \ref{le-ksvd} that the Frobenius norm of $E_1$ is
\begin{equation}\label{eq:e1norm}
	\|E_1\|_F = \sqrt{\sum_{i=2}^q\tilde{\sigma}^2_i},
\end{equation}
where $q=\min\{s_1s_2,(n_1n_2)/(s_1s_2)\}$ and $\tilde{\sigma}_2, \dots, \tilde{\sigma}_q$ are all the singular values of $\widetilde{A}$, defined by \cref{eq:RA}, except the largest one.

Next, the truncated SVD-STP can be obtained by replacing the full SVD of $B$ with the truncated SVD of $B$ in the proof of Theorem \ref{thm:fullsvd}. Consider the similarity, its proof is omitted.

\begin{theorem}[Truncated SVD-STP]\label{thm:truncatedsvd} 
	Let $A\in \IR^{n_1\times n_2}$ and $s_1$, $s_2$ be factors of $n_1$, $n_2$, respectively. Take any positive integer $r\leq p$ where the integer $p:=\min\{n_1/s_1, n_2/s_2\}$. Suppose that $A = U \ltimes \Sigma \ltimes V^\top + E_1$ is the SVD-STP of $A$. Then the truncated SVD-STP of $A$ is given by
	\begin{equation}\label{eq:gtsvd}
		A = U_r\ltimes \Sigma_r \ltimes V_r^\top + E_2 =\Sigma_r \ltimes_1 U_r\ltimes_2 V_r + E_2,
	\end{equation}
	where $E_2$ is the approximation error, $U_r\in \IR^{(n_1/s_1)\times r}$ and $V_r\in \IR^{(n_2/s_2)\times r}$ contain respectively only the first $r$ columns of the orthogonal matrices $U$ and $V$, and $\Sigma_r=\blkdiag (S_{1}, S_{2},\dots, S_{r})\in \IR^{rs_1\times rs_2}$ contains only the first $r$ diagonal blocks of $\Sigma$ with each block $S_{k}$ being an $s_1$-by-$ s_2$ matrix for $1\leq k\leq r$ and $\|S_{1}\|_F\geq \|S_{2}\|_F \geq \cdots \geq \|S_{r}\|_F$. 
\end{theorem}

Algorithm \ref{alg:TSVD-STP} given for computing the approximation \cref{eq:gtsvd} is similar to Algorithm~\ref{alg:FSVD-STP}. Just use the truncated SVD of $B$ instead of the full version.

\begin{algorithm}
	\caption{Truncated SVD-STP}
	\label{alg:TSVD-STP}
	\begin{algorithmic}
		\State Given $\CA\in \IR^{n_1\times n_2}$, factors $s_1$, $s_2$, and an integer $r$.
		
		\State (1) Compute matrices $B\in \IR^{(n_1/s_1)\times (n_2/s_2)}$ and $C\in \IR^{s_1\times s_2}$, which minimize $\|A-B\otimes C\|_F$, by using Lemma~\ref{le-ksvd}.
		
		\State (2) Compute the truncated SVD $B=U_r\Sigma_{B,r} V^\top_r$, where $U_r\in \IR^{(n_1/s_1)\times r}$, $V_r\in \IR^{(n_2/s_2)\times r}$, and $\Sigma_{B,r}\in\IR^{r\times r}$ is a diagonal matrix with the singular values of $B$ on the diagonal.
		
		\State (3) Compute $\Sigma_r=\Sigma_{B,r}\otimes C\in \IR^{rs_1\times rs_2}$ which is a rectangular block-diagonal matrix with each block being an $s_1$-by-$s_2$ matrix. 
		
		\Return $\Sigma_r, U_r, V_r$
	\end{algorithmic}
\end{algorithm}

Considering that the truncated SVD-STP is an approximation of the given matrix, we propose an upper bound for the approximation error.

\begin{theorem}[Upper bound for approximation error]
	With the same assumptions and symbols as in Theorem \ref{thm:truncatedsvd}, we have the following upper bound for the approximation error $E_2$:
	\begin{equation}\label{eq:upb1}
		\|E_2\|_F \leq \sqrt{\sum_{i=2}^q\tilde{\sigma}^2_i} + \sqrt{\sum_{i=r+1}^p \|S_{i}\|_F^2},
	\end{equation}
	where $q=\min\{s_1s_2,(n_1n_2)/(s_1s_2)\}$ and $\tilde{\sigma}_2, \dots, \tilde{\sigma}_q$ are all the singular values of $\widetilde{A}$, defined by \cref{eq:RA}, except the largest one.
\end{theorem}

\begin{proof}
 Let $\widecheck{A} = U\ltimes \Sigma \ltimes V^\top$ be the SVD-STP of $A$. For the approximation error $E_2$ in Theorem \ref{thm:truncatedsvd}, we have
\begin{align*}
	\|E_2\|_F & = \big\|A-U_r \ltimes \Sigma_r \ltimes V_r^\top \big\|_F \\
	& = \big\|A-\widecheck{A}+\widecheck{A} -U_r \ltimes \Sigma_r \ltimes V_r^\top \big\|_F \\
	& \leq  \big\|A-\widecheck{A} \,\big\|_F + \big\|\widecheck{A} -U_r \ltimes \Sigma_r \ltimes V_r^\top \big\|_F. 
\end{align*}

For the first part $\big\|A-\widecheck{A}\,\big\|_F$, it follows from the proof of Theorem \ref{thm:fullsvd} and \cref{eq:e1norm} that
\begin{equation*}
	\big\|A-\widecheck{A}\,\big\|_F  =\big\|A -U \ltimes \Sigma \ltimes V^\top \big\|_F 
	= \|A-B\otimes C\|_F = \|E_1\|_F= \sqrt{\sum_{i=2}^q\tilde{\sigma}^2_i},
\end{equation*}
where $q=\min\{s_1s_2,(n_1n_2)/(s_1s_2)\}$ and $\tilde{\sigma}_2, \dots, \tilde{\sigma}_q$ are all the singular values of $\widetilde{A}$, defined by \cref{eq:RA}, except the largest one.

As for the second part $\big\|\widecheck{A} -U_r \ltimes \Sigma_r \ltimes V_r^\top \big\|_F$, since $U_r$ and $V_r$ contain respectively only the first $r$ columns of the orthogonal matrices $U$ and $V$ and $\Sigma_r$ contains only the first $r$ diagonal blocks of $\Sigma$, we have from the orthogonal invariance of the Frobenius norm that 
\begin{align*}
	& \big\|\widecheck{A} -U_r \ltimes \Sigma_r \ltimes V_r^\top \big\|_F  = \big\|U \ltimes \Sigma \ltimes V^\top-U_r \ltimes \Sigma_r \ltimes V_r^\top \big\|_F\\
	= & \big\|U \ltimes \Sigma \ltimes V^\top-U \ltimes \blkdiag (S_{1},\dots, S_{r},O,\dots,O) \ltimes V^\top \big\|_F\\
	= & \big\|(U\otimes I_{s_1}) \blkdiag (O, \dots,O, S_{r+1},\dots,S_{p}) (V \otimes I_{s_2})^\top \big\|_F\\
	= & \|\blkdiag (O, \dots,O, S_{r+1},\dots,S_{p})\|_F = \sqrt{\sum_{i=r+1}^p \|S_{i}\|_F^2}.
\end{align*}

Combining the above two parts gives the bound \cref{eq:upb1} for $\|E_2\|_F$.
\end{proof}

\subsection{The higher-order case: HOSVD}

Next the higher-order case is considered. Given a tensor $\CA\in \IR^{n_1\times n_2\times \cdots \times n_d}$ and a target rank $\br = (r_1,\dots, r_d)$, the idea of the Tucker approximation is to find a ``core tensor'' $\CB\in \IR^{r_1\times r_2\times \cdots \times r_d}$ and matrices $U^{(i)}\in \IR^{n_i \times r_i}$ (usually orthogonal) for $1\leq i\leq d$, such that
\begin{equation*}
	\CA= \CB \times_1 U^{(1)}\times_2 U^{(2)}\times_3 \cdots \times_d U^{(d)}+\CE,
\end{equation*}
where the tensor $\CE\in \IR^{n_1\times n_2\times \cdots \times n_d}$ represents the approximation error. This leads to the following optimization problem:
\begin{equation} \begin{split}
		\text{minimize} \ \ & \ \ \|\CA-\CB \times_1 U^{(1)}\times_2 U^{(2)}\times_3 \cdots \times_d U^{(d)}\|_F \\
		\text{subject to}\ \  & \ \ U^{(i)\top} U^{(i)} = I \ \ \text{for}\ \ i=1,\dots, d.
\end{split} \end{equation}
Such a best rank-$(r_1, r_2,\dots, r_d)$ approximation of a $d$th-order tensor can be derived by the HOOI algorithm \cite{KrLe80,LaMoVa20_2}. The HOSVD can serve as a good starting point for the HOOI algorithm.

The following two theorems are respectively the HOSVD and the truncated HOSVD based on the STP, and Algorithms \ref{alg:fhosvd} and \ref{alg:thosvd} are given to construct these two Tucker-like approximations. For two collections of positive integers $n_1, n_2, \dots, n_d$ and $s_1, s_2, \dots, s_d$, and $1\leq k\leq d$, let $s_k$ be a factor of $n_k$. In the following content, we denote 
\begin{equation}\label{eq:nkskpk}
	n_k^\prime := \frac{n_1n_2\cdots n_d}{n_k},\  s_k^\prime := \frac{s_1s_2\cdots s_d}{s_k},\ \text{and} \ p_k := \min\Big\{\frac{n_k}{s_k}, \frac{n_k^\prime}{s_k^\prime}\Big\}
\end{equation}
for $1\leq k \leq d$.

\begin{theorem}[HOSVD-STP]\label{thm:full-hosvd-stp}
	Suppose that $\CA\in \IR^{n_1\times n_2\times \cdots \times n_d}$ and $s_k$ is a factor of $n_k$ for $1\leq k\leq d$. Then there exists a sequence of orthogonal matrices $U^{(k)}\in \IR^{(n_k/s_k)\times (n_k/s_k)}$ for $1\leq k\leq d$, such that 
	\begin{equation}
		\CA = \CB \ltimes_1 U^{(1)}\ltimes_2 U^{(2)}\ltimes_3 \cdots \ltimes_d U^{(d)},
	\end{equation}
	where $\CB\in \IR^{n_1\times n_2\times \cdots \times n_d}$ has the property that its mode-$k$ unfolding $\CB_{(k)}$ is a matrix with orthogonal row blocks for all possible values of $k$. More precisely, there exists a sequence of matrices $S^{(k)}_1,S^{(k)}_2,\dots, S^{(k)}_{p_k}\in \IR^{s_k\times s_k}$ such that
	\begin{equation}\label{eq:BkBkT}
		\CB_{(k)}\CB_{(k)}^\top \approx \blkdiag (S^{(k)}_1 S^{(k)\top}_1,S^{(k)}_2 S^{(k)\top}_2,\dots, S^{(k)}_{p_k} S^{(k)\top}_{p_k})\in \IR^{n_k\times n_k}
	\end{equation}
	with
	$$
	\|S^{(k)}_1\|_F\geq \|S^{(k)}_2\|_F \geq \cdots \geq \|S^{(k)}_{p_k}\|_F,
	$$
	where $p_k$ is defined by \cref{eq:nkskpk} for all $1\leq k \leq d$.
\end{theorem}

\begin{proof}
Given $\CA\in \IR^{n_1\times n_2\times \cdots \times n_d}$, let $s_k$ be a factor of $n_k$ for all $1\leq k\leq d$. Consider the SVD-STP (see Theorem \ref{thm:fullsvd}) of the mode-$k$ unfolding $\CA_{(k)}$:
\begin{equation}\label{eq:AkSVDSTP}
	\CA_{(k)} = U^{(k)} \ltimes \Sigma^{(k)} \ltimes V^{(k)\top} + E_1^{(k)}=\widehat{U^{(k)}}\Sigma^{(k)} \widehat{V^{(k)}}^\top+ E_1^{(k)},
\end{equation}
where $E_1^{(k)}$ is the mode-$k$ unfolding approximation error, $\widehat{U^{(k)}}=U^{(k)} \otimes I_{s_k}$ and $\widehat{V^{(k)}}=V^{(k)} \otimes I_{s_k^\prime}$ are orthogonal matrices, and 
$$
\Sigma^{(k)}=\blkdiag (S^{(k)}_1,S^{(k)}_2,\dots, S^{(k)}_{p_k}) \in \IR^{n_k\times n_k^\prime}
$$
with
$$
\|S^{(k)}_1\|_F\geq \|S^{(k)}_2\|_F \geq \cdots \geq \|S^{(k)}_{p_k}\|_F
$$
and $n_k^\prime$, $s_k^\prime$, and $p_k$ defined by \cref{eq:nkskpk} for all $1\leq k \leq d$. Denote
\begin{equation}\label{eq:BA}
	\CB = \CA \ltimes_1 U^{(1)\top} \ltimes_2 U^{(2)\top}\ltimes_3 \cdots \ltimes_d U^{(d)\top}
\end{equation}
and 
$$ 
\Pi^{(-k)}=\widehat{U^{(d)}}\otimes \cdots \otimes \widehat{U^{(k+1)}}\otimes \widehat{U^{(k-1)}} \otimes \cdots \otimes \widehat{U^{(1)}}.
$$
Using Propositions \ref{pro:T} and \ref{pro:unfold},  (\ref{eq:BA}) can be rewritten as in its mode-$k$ unfolding
\begin{equation}\label{eq:Ak1}
	\CB_{(k)} = \widehat{U^{(k)}}^\top \CA_{(k)} \Pi^{(-k)}\in \IR^{n_k \times n_k^\prime}.
\end{equation}
Combining \cref{eq:AkSVDSTP} and \cref{eq:Ak1} gives
\begin{equation*}
	\CB_{(k)} = \Sigma^{(k)} \widehat{V^{(k)}}^\top \Pi^{(-k)} + \widehat{U^{(k)}}^\top E_1^{(k)}\Pi^{(-k)}.
\end{equation*}
By omitting the error part $\widehat{U^{(k)}}^\top E_1^{(k)}\Pi^{(-k)}$, it follows from the orthogonality of $\widehat{V^{(k)}}$ and $\Pi^{(-k)}$ that
\begin{align*}
	& \CB_{(k)}\CB_{(k)}^\top  \approx \Sigma^{(k)} \Sigma^{(k)\top}\\
	= & \blkdiag (S^{(k)}_1,S^{(k)}_2,\dots, S^{(k)}_{p_k}) \blkdiag (S^{(k)\top}_1,S^{(k)\top}_2,\dots, S^{(k)\top}_{p_k})\\
	= & \blkdiag (S^{(k)}_1 S^{(k)\top}_1,S^{(k)}_2 S^{(k)\top}_2,\dots, S^{(k)}_{p_k} S^{(k)\top}_{p_k})
	\in \IR^{n_k\times n_k}
\end{align*}
with $ \|S^{(k)}_1\|_F\geq \|S^{(k)}_2\|_F \geq \cdots \geq \|S^{(k)}_{p_k}\|_F$ for all $1\leq k \leq d$, and we can write 
\begin{equation*}
	\CA = \CB \ltimes_1 U^{(1)}\ltimes_2 U^{(2)}\ltimes_3 \cdots \ltimes_d U^{(d)}
\end{equation*}
by Proposition \ref{pro:T} and the orthogonality of $U^{(k)}$ for $1 \leq k \leq d$.
\end{proof}

It should be noted that, although HOSVD is a higher-order generalization of SVD, in Theorem \ref{thm:full-hosvd-stp}, the conclusion when $d=2$ is slightly different from the result of Theorem \ref{thm:fullsvd}.

\begin{algorithm}
	\caption{HOSVD-STP}
	\label{alg:fhosvd}
	\begin{algorithmic}
		\State Given a tensor $\CA\in \IR^{n_1\times n_2\times \cdots \times n_d}$, $n=n_1n_2\cdots n_d$, and a factor vector $\bs$=$(s_1, s_2, \dots, s_d)$.
		
		\For{ $k = 1,\dots, d$}
		
		\State Compute the factor $U\in \IR^{(n_k/s_k)\times (n_k/s_k)}$ of $\CA_{(k)}\in \IR^{n_k\times (n/n_k)}$ in the decomposition (\ref{eq:gsvd}) by Algorithm~\ref{alg:FSVD-STP}.
			
		\State $U^{(k)} \leftarrow U$
		
		\EndFor
		
		\State $\CB \leftarrow \CA \ltimes_1 U^{(1)\top}\ltimes_2 U^{(2)\top}\ltimes_3 \cdots \ltimes_d U^{(d)\top}$ 
		
		\Return $\CB, U^{(1)}, U^{(2)}, \dots, U^{(d)}$
	\end{algorithmic}
\end{algorithm}


The following theorem is the truncated version of HOSVD-STP. Since its proof is similar to the one of Theorem \ref{thm:full-hosvd-stp}, we omit it.

\begin{theorem}[Truncated HOSVD-STP]\label{thm:trun-hosvd-stp}
	Given $\CA\in \IR^{n_1\times n_2\times \cdots \times n_d}$, for $1\leq k\leq d$, let $s_k$ be a factor of $n_k$, and take any positive integer $r_k\leq p_k$ where $p_k$ is defined by \cref{eq:nkskpk}. Suppose that $	\CA = \CB \ltimes_1 U^{(1)}\ltimes_2 U^{(2)}\ltimes_3 \cdots \ltimes_d U^{(d)}$ is the HOSVD-STP of $\CA$. Then the truncated HOSVD-STP of $\CA$ is given by
	\begin{equation}\label{eq:trunHOSVD-STP}
		\CA = \CB_{\br} \ltimes_1 U_{r_1}^{(1)}\ltimes_2 U_{r_2}^{(2)}\ltimes_3 \cdots \ltimes_d U_{r_d}^{(d)} + \CE_1,
	\end{equation}
	where $\CE_1$ represents the approximation error, $U^{(k)}_{r_k}\in \IR^{(n_k / s_k)\times r_k}$ contains only the first $r_k$ columns of $U^{(k)}$, and
	\begin{equation}
		\CB_{\br (k)}\CB_{\br (k)}^\top \approx \blkdiag (S^{(k)}_1 S^{(k)\top}_1,S^{(k)}_2 S^{(k)\top}_2,\dots, S^{(k)}_{r_k}S^{(k)\top}_{r_k})\in \IR^{r_ks_k\times r_ks_k}
	\end{equation}
	contains only the first $r_k$ diagonal blocks of $\CB_{(k)}\CB_{(k)}^\top$ with each diagonal block $S^{(k)}_i S^{(k)\top}_i$ being an $s_k$-by-$s_k$ matrix and 
	$$
	\|S^{(k)}_1\|_F\geq \|S^{(k)}_2\|_F \geq \cdots \geq \|S^{(k)}_{r_k}\|_F
	$$
	for all $1\leq k \leq d$.
\end{theorem}

\begin{algorithm}
	\caption{Truncated HOSVD-STP}
	\label{alg:thosvd}
	\begin{algorithmic}
		\State Given $\CA\in \IR^{n_1\times n_2\times \cdots \times n_d}$, $n=n_1n_2\cdots n_d$, a factor vector $\bs$=$(s_1, s_2, \dots, s_d)$, and a vector $\br=(r_1, r_2, \dots, r_d)$.
		
		\For{$k = 1,\dots, d$}
		
		\State Compute the factor $U_{r_k}\in \IR^{(n_k/s_k)\times r_k}$ of $\CA_{(k)}\in \IR^{n_k\times (n/n_k)}$ in the decomposition (\ref{eq:gtsvd}) by Algorithm~\ref{alg:TSVD-STP}.
		
		\State $U_{r_k}^{(k)} \leftarrow U_{r_k}$
		
		\EndFor
		
		\State $\CB_{\br} \leftarrow \CA \ltimes_1 U_{r_1}^{(1)\top}\ltimes_2 U_{r_2}^{(2)\top}\ltimes_3 \cdots \ltimes_d U_{r_d}^{(d)\top}$ \vspace{0.2cm}
		
		\Return $\CB_{\br}, U_{r_1}^{(1)}, U_{r_2}^{(2)}, \dots, U_{r_d}^{(d)}$
	\end{algorithmic}
\end{algorithm}

In Theorem \ref{thm:trun-hosvd-stp}, the approximation $\CA \approx \CB_{\br} \ltimes_1 U_{r_1}^{(1)}\ltimes_2 U_{r_2}^{(2)}\ltimes_3 \cdots \ltimes_d U_{r_d}^{(d)}$ can also be represented as
\begin{equation}
	\mathcal{A}\approx \mathcal{A}\ltimes_1(U_{r_1}^{(1)}U_{r_1}^{(1)\top})\ltimes_2(U_{r_2}^{(2)}U_{r_2}^{(2)\top})\ltimes_3\dots\ltimes_d (U_{r_d}^{(d)}U_{r_d}^{(d)\top}),
\end{equation}
where $\mathcal{A}\ltimes_k (U_{r_k}^{(k)}U_{r_k}^{(k)\top})$ can be seen as a STP-based multilinear orthogonal projection \cite{VaVaMe12} along mode $k$. By the similar error analysis in \cite[Section 5]{VaVaMe12}, we can obtain an upper bound for the approximation error $\CE_1$ in Theorem \ref{thm:trun-hosvd-stp}.

\begin{theorem}[Upper bound for approximation error]
	With the same assumptions and symbols as in Theorem \ref{thm:trun-hosvd-stp}, we have the following upper bound for the approximation error $\CE_1$:
	\begin{equation}
		\|\CE_1\|_F \leq \sqrt{\sum_{k=1}^{d}\sum_{i=2}^{q_k}\tilde{\sigma}^{(k)2}_i} + \sqrt{\sum_{k=1}^{d}\sum_{i=r_k+1}^{p_k} \|S_{i}^{(k)}\|_F^2},
	\end{equation}
	where $q_k=\min\{s_ks_k^\prime,(n_kn_k^\prime)/(s_ks_k^\prime)\}$ and $\tilde{\sigma}_2^{(k)}, \dots, \tilde{\sigma}_{q_k}^{(k)}$ are all the singular values of $\widetilde{\CA}_{(k)}$, defined by \cref{eq:RA}, except the largest one.
\end{theorem}


%
%

\section{Experimental results}
\label{sec:experiments}
In this section, some numerical examples are given to illustrate our theoretical results. We use the MATLAB language to do the numerical experiments.

{\bf Example 1} \ \   Let $A\in \IR^{n\times n}$ and $s$ be a factor of $n$. In this example, the matrix $A$ is generated by MATLAB function {\tt rand$(n,n)$}. We compare our full SVD-STP (FSVD-STP) and truncated SVD-STP (TSVD-STP) with the conventional truncated SVD (TSVD, taking the $r$ largest singular values and their corresponding singular vectors) from the time and their relative errors $\|A-\tilde{A}\|_F/\|A\|_F$, where $\tilde{A}$ is the computed approximation. The comparison results are shown in Table~\ref{ta1}.

\begin{table}[h]
	\renewcommand{\arraystretch}{1.5}
	\caption{The comparison results for FSVD-STP, TSVD-STP, and TSVD }\vspace{0.2cm}
	\centering {
		\begin{tabular}[l]{@{}c|c|c|c|c|c}
			\hline
			$n$& $s,r$ & Average value &  FSVD-STP  & TSVD-STP & TSVD \\\hline   
			
			\multirow{2}*{5000} & \multirow{2}*{2,50} &  Time &  4.72s & 3.43s & 13.45s \\
			&& Relative errors  & 0.4330 & 0.4954 & 0.4905 \\\hline
			
			\multirow{2}*{10000} & \multirow{2}*{2,50} & Time &  25.59s & 17.05s & 66.75s \\
			&& Relative errors   & 0.4331 & 0.4977 & 0.4952 \\\hline
			
			\multirow{2}*{10000} & \multirow{2}*{5,50} & Time &   35.97s & 34.84s & 71.16s \\
			&& Relative errors   & 0.4899 & 0.4991 & 0.4952  \\\hline
			
			\multirow{2}*{10000} & \multirow{2}*{10,50} & Time &  16.89s & 16.43s & 59.72s \\
			&& Relative errors   & 0.4975 & 0.4996& 0.4952 \\\hline
			
			\multirow{2}*{10000} & \multirow{2}*{20,50} & Time &  22.38s & 22.02s & 64.96s \\
			&& Relative errors   & 0.4994 & 0.4999& 0.4952 \\\hline
			
			\multirow{2}*{20000} & \multirow{2}*{10,50} & Time &  59.36s & 54.81s & 319.55s \\
			&& Relative errors   & 0.4975 & 0.4998& 0.4976 \\\hline
			
			\multirow{2}*{40000} & \multirow{2}*{10,50} & Time &  232.74s & 220.60s & 1716.35s \\
			&& Relative errors   & 0.4975 & 0.4999& 0.4988 \\\hline
			
		\end{tabular}\label{ta1}}
\end{table}

From Table~\ref{ta1}, we see that as the increasing of the matrix size $n$, our algorithms FSVD-STP and TSVD-STP take much less time than the conventional truncated SVD with approximate relative errors. Besides, for the required storage, the original matrix $A$ requires $n^2$ numbers. Storing the result of the conventional SVD with rank-$r$ requires $2nr+r$ numbers. This is a big savings if $r$ is much smaller than $n$. While storing the results of FSVD-STP and TSVD-STP only require $2n^2/s^2+n/s+s^2$ and $2nr/s+r+s^2$ numbers, respectively, which could be less than the one of the conventional SVD if $s$ is a small factor of $n$ (according to our numerical experiments, $s\leq 10$ would be a better choice). 

\begin{figure}[h!]
	\begin{center}
		\includegraphics[width=0.75\textwidth]{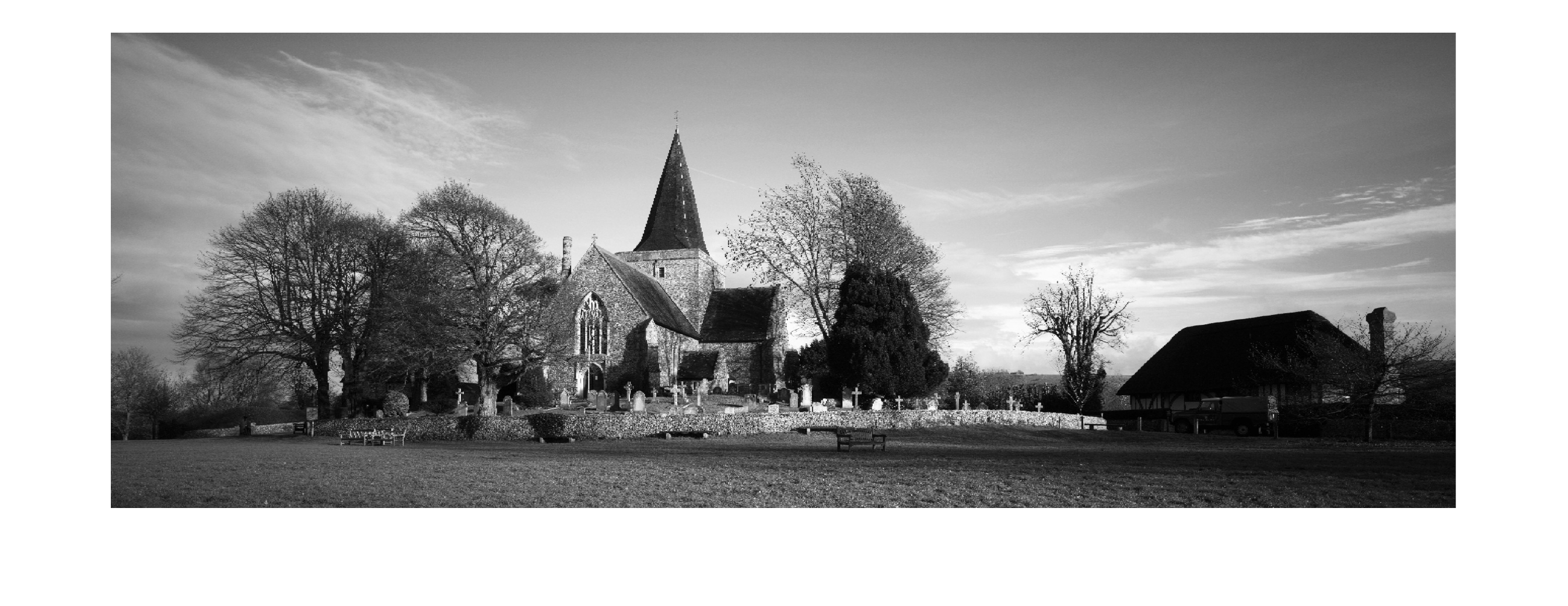}\\[-1.1cm]
		\caption*{The original image}
		\includegraphics[width=0.75\textwidth]{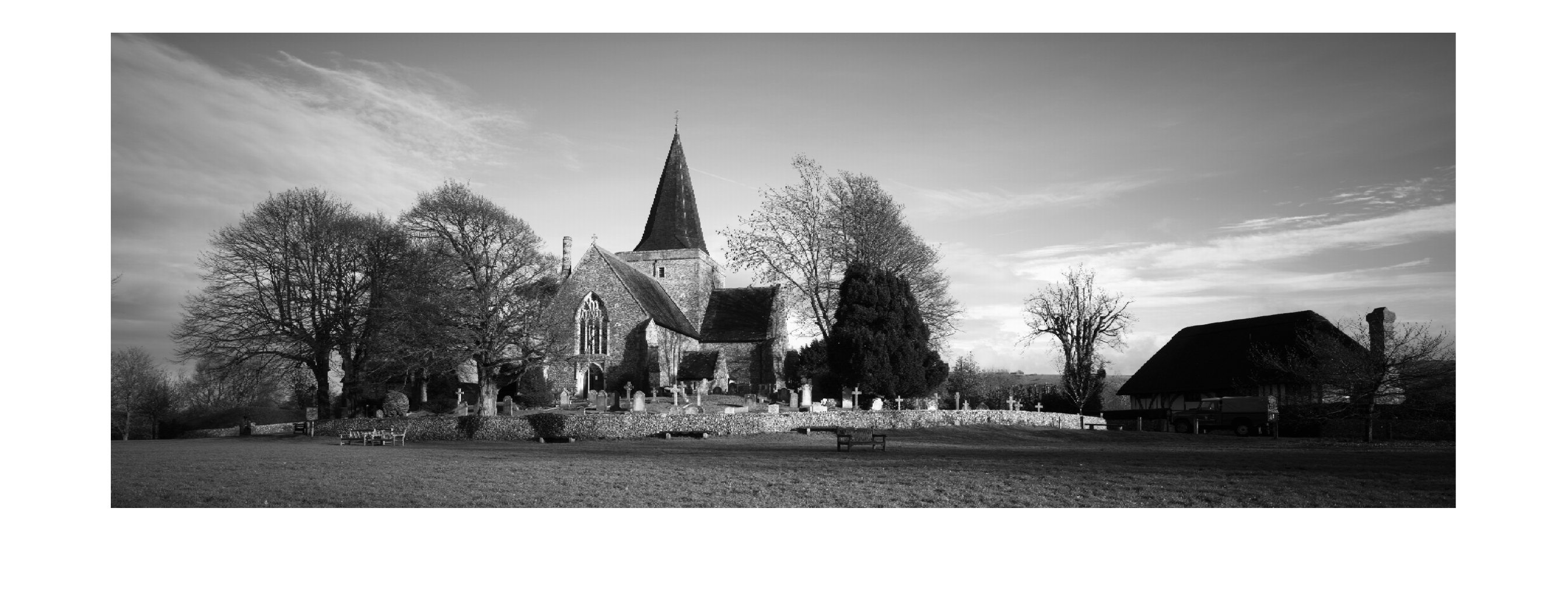}\\[-1.1cm]
		\caption*{Restoration by the FSVD-STP with $s_1=2, s_2=5$ }
		\includegraphics[width=0.75\textwidth]{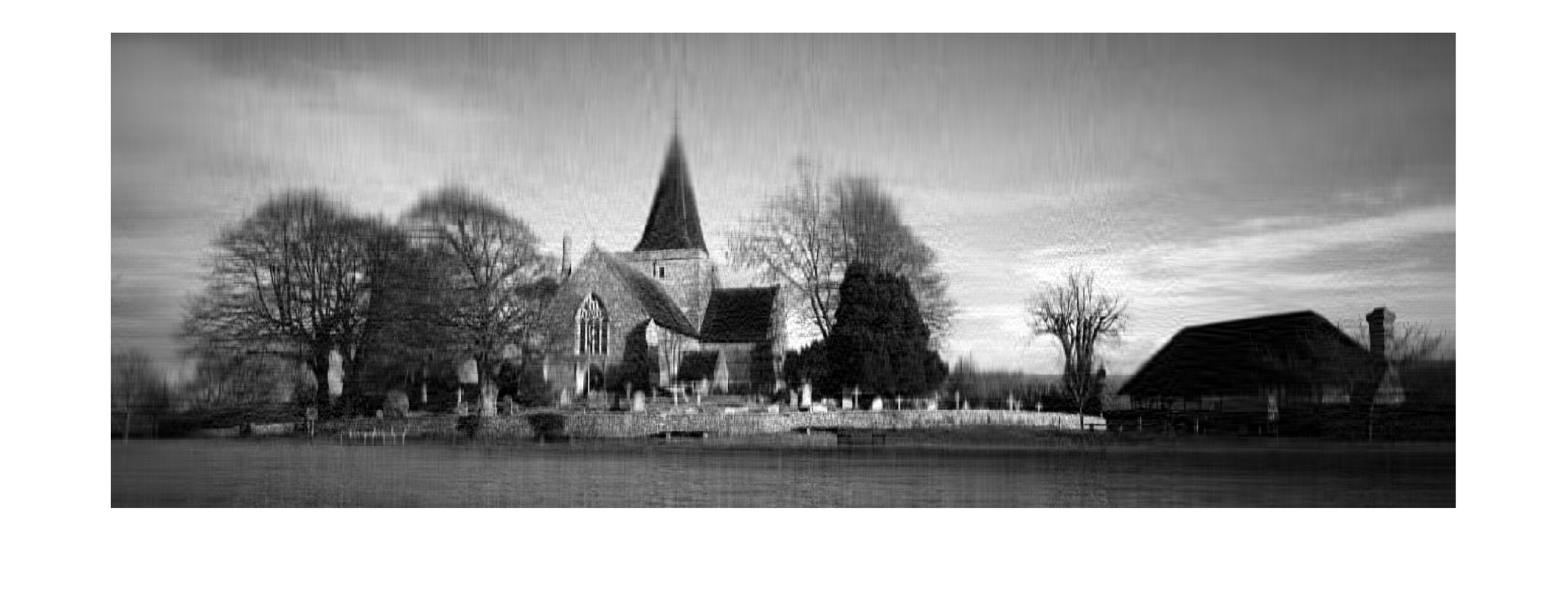}\\[-1.1cm]
		\caption*{Restoration by the TSVD-STP with $s_1=2, s_2=5$, and $r=50$ }
		\includegraphics[width=0.75\textwidth]{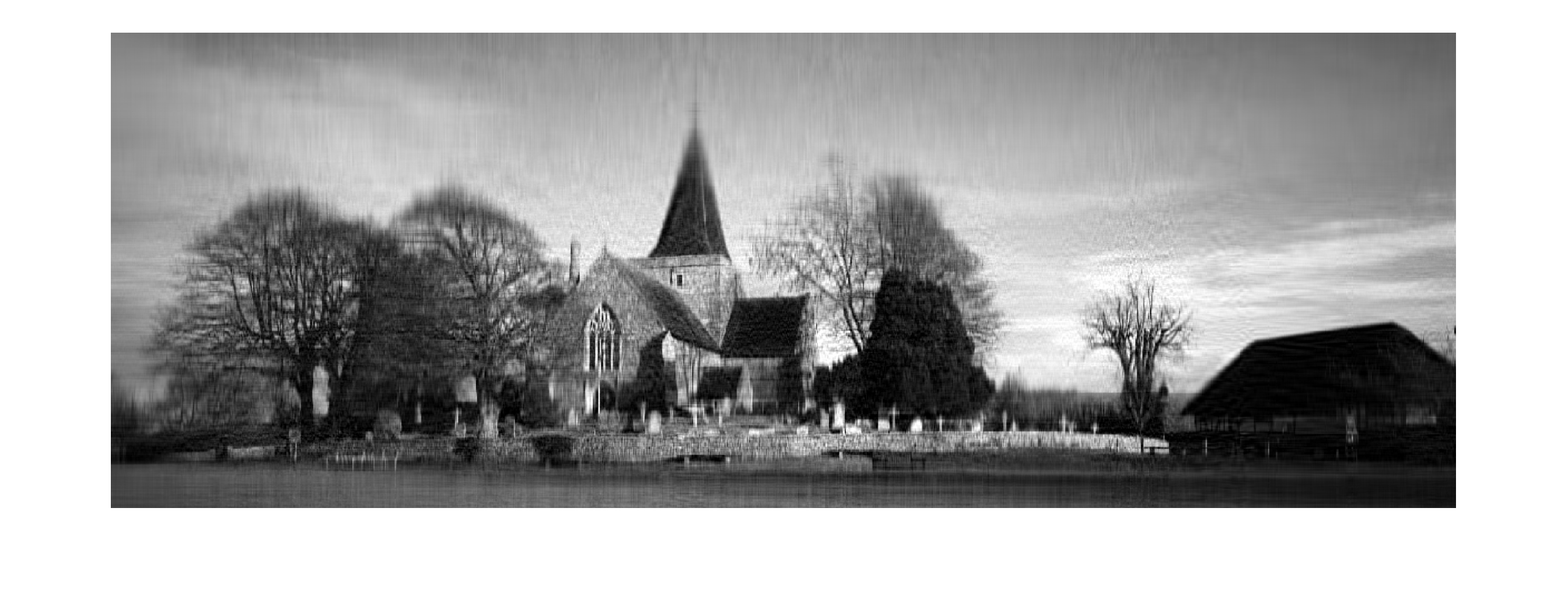}\\[-1.1cm]
		\caption*{Restoration by the TSVD with $r=50$ }
		\caption{The image restoration comparison }\label{fig:ex-pic}
	\end{center}	
\end{figure}

{\bf Example 2} \ \  In this example, an image whose corresponding matrix $A$ be of the size of $7500\times 21250$ is given. We compare the image compression results by three methods --- FSVD-STP, TSVD-STP, and TSVD, for the given image. The restoration results are shown in Figure \ref{fig:ex-pic}. The first image is the original image. The second image is the restoration of the original one by the FSVD-STP method with $s_1=2, s_2=5$. The third image is another restoration of the first image by the TSVD-STP method with $s_1=2, s_2=5$, and $r=50$. The last one is the restoration of the original image by the TSVD with $r=50$.

\begin{table}[th]
	\renewcommand{\arraystretch}{1.5}
	\caption{The comparison results in image restoration }\vspace{0.2cm}
	\centering {
		\begin{tabular}[l]{@{}c|c|c|c}
			\hline
			Average value & FSVD-STP   & TSVD-STP & TSVD \\\hline   
			
			Time &  33.2727s & 21.7582s & 45.4027s \\
			Relative errors  & 0.0971 & 0.1518 & 0.1471 \\
			SSIM  & 0.8956 & 0.7346 & 0.7400 \\
			PSNR  &  27.0951 &  23.3984 & 23.6674 \\\hline
			
		\end{tabular}\label{image-ta}}
\end{table}

{\bf Example 3} \ \   Suppose that $\CA\in \IR^{n\times n\times \cdots \times n}$ is a $d$th-order tensor and $\bs=(s, s, \dots, s)$ is a $d$-dimensional vector, where $s$ is a factor of $n$. In this example, $\CA$ is generated by MATLAB function {\tt rand$(n,n,\dots,n)$}. We compare the truncated HOSVD-STP (THOSVD-STP) with the truncated HOSVD (THOSVD with rank-$(r,r,\dots,r)$, MATLAB function ``{\tt hosvd}'' provided by \cite{BaKo06,BaKo19}) from the time and their relative errors $\|\CA-\tilde{\CA}\|_F/\|\CA\|_F$, where $\tilde{\CA}$ is the computing approximation. The comparison results are shown in Table~\ref{ta2}.

\begin{table}[htb]
	\renewcommand{\arraystretch}{1.5}
	\caption{The comparison results for FHOSVD-STP, THOSVD-STP, and THOSVD }\vspace{0.2cm}
	\centering {
		\begin{tabular}[l]{@{}c|c|c|c|c|c}
			\hline
			$n,d$& $s,r$ & Average value &  FHOSVD-STP & THOSVD-STP & THOSVD \\\hline
			\multirow{2}*{100,3} & \multirow{2}*{2,20} &  Time &  0.03s & 0.02s & 0.03s \\
			&& Relative errors & 1.1441e-12 &  0.4822 & 0.4950 \\\hline
			
			\multirow{2}*{500,3} & \multirow{2}*{2,20} & Time &  3.31s &  1.27s & 1.44s \\
			&& Relative errors  & 4.7769e-11 &  0.4999 & 0.4999 \\\hline
			
			\multirow{2}*{500,3} & \multirow{2}*{5,20} & Time & 3.68s & 1.64s & 1.30s\\
			&& Relative errors & 0.0030  &  0.4980 & 0.4999 \\\hline
			
			\multirow{2}*{100,4} & \multirow{2}*{2,20} & Time & 1.83s & 0.73s & 1.07s\\
			&& Relative errors & 1.5811e-11  &  0.4935 & 0.4995 \\\hline
			
			\multirow{2}*{200,4} & \multirow{2}*{2,20} & Time & 100.42s & 12.81s & 37.30s\\
			&& Relative errors & 1.1394e-10  &  0.4996 & 0.4999 \\\hline

			\multirow{2}*{200,4} & \multirow{2}*{5,20} & Time & 105.24s & 19.44s & 34.20s\\
			&& Relative errors &  0.0045  &  0.4841 & 0.5000 \\\hline
			
			\multirow{2}*{50,5} & \multirow{2}*{2,10} & Time & 7.40s &   2.73s & 3.42s \\
			&& Relative errors & 1.2687e-11  & 0.4974 & 0.4999 \\\hline
			
			\multirow{2}*{20,6} & \multirow{2}*{2,5} & Time & 1.85s &  0.70s & 0.79s \\
			&& Relative errors & 7.9828e-12  & 0.4961 & 0.4999 \\\hline
			
		\end{tabular}\label{ta2}}
\end{table}

From Table~\ref{ta2}, we see that the truncated HOSVD-STP could be more efficient than the conventional truncated HOSVD for some cases. As for the required storage, the original tensor $\CA$ requires $n^d$ numbers. Storing the result of the conventional truncated HOSVD with rank-$(r,r,\dots,r)$ requires $r^d+dnr$ numbers, while storing the result of THOSVD-STP with the same $r$ requires $(rs)^d+dnr/s$ numbers, which could be less than the one of the conventional truncated HOSVD if $r$ and $s$ are much small than $n$, or more precisely if 
$$
n\geq \frac{r^{d-1}s(s^d-1)}{d(s-1)}.
$$

\section{Concluding remarks}
\label{sec:con}
A Tucker-like approximation based on the STP was considered. We developed new versions of SVD and HOSVD for matrices and higher-order tensors, respectively. Numerical experiments showed that the new SVD-STP and HOSVD-STP could be better than the conventional SVD and HOSVD for large-scale dimensional problems by choosing appropriate factor parameters. It is hard to give a univeral method for choosing the factors $s_1, s_2,\dots, s_d$ in all computing problems. We have done many numerical experiments. We found that the small factors ($\leq 10$) would be the better choice for most large-scale computing problems.


%





\nocite{*}
\bibliographystyle{unsrtnat}
\bibliography{references.bib}%



\end{document}